\documentclass[11pt]{amsart}

\usepackage{amsmath,amssymb,amsthm,mathtools}
\usepackage[hidelinks]{hyperref}
\usepackage[margin=1.10in]{geometry}

\newtheorem{theorem}{Theorem}[section]
\newtheorem{proposition}[theorem]{Proposition}
\newtheorem{lemma}[theorem]{Lemma}
\newtheorem{corollary}[theorem]{Corollary}
\newtheorem{definition}[theorem]{Definition}
\theoremstyle{remark}
\newtheorem{remark}[theorem]{Remark}
\newtheorem{problem}[theorem]{Problem}

\newcommand{\R}{\mathbb R}

\newcommand{\Span}{\operatorname{span}}
\newcommand{\relint}{\operatorname{relint}}
\newcommand{\St}{\operatorname{St}}
\newcommand{\pos}{\operatorname{pos}}
\newcommand{\Ext}{\operatorname{Ext}}
\newcommand{\PP}{\mathbb P}

\title[Affine anchors in dimension three]
{Affine Anchors and Cylinder Obstructions in the Three-Dimensional Tingley Problem}
\author{Yicen Ma}
\date{September 15, 2026}

\begin{document}

\begin{abstract}
Let $X$ be a three-dimensional real Banach space and let
$f:S_X\to S_Y$ be a surjective isometry.  We study the propagation of
an affine formula for $f$ on a relatively open part of $S_X$.  A
finite family of antipodal distance coordinates propagates a linear
anchor except at three explicitly described degeneracies: a facet, an
open face-star, or a family of chord cones with a common cylindrical
kernel.  We prove that, if the norm has no fixed-direction cylindrical
open cone, every affine open anchor is linear and global.  This yields
a source-geometric criterion for the Mazur--Ulam property and covers,
among other non-strictly-convex examples, the Euclidean double cone.
We also give a segment-saturation criterion which proves the property
for every prism $Z\oplus_\infty\R$, where $Z$ is an arbitrary real
Banach plane.  Finally, in the remaining cylindrical regime, we prove
that a chord-saturated same-kernel network cannot be confined to one
proper projective quotient arc.  We then show that this multi-arc
difficulty and the transverse-ruled part of producing an initial affine
anchor reach the same final obstruction: upgrading norm calibration on
an ambient open cone, together with calibration on a few spherical
segments, to pointwise calibration of the sphere map on a
two-dimensional open patch.
\end{abstract}

\maketitle

\begin{center}
\fbox{\parbox{0.88\textwidth}{\small
\textbf{Draft status.} This manuscript records proof candidates that
have passed internal consistency checks but not an independent proof
audit.  It does not claim the general three-dimensional Tingley
problem.  Section~\ref{sec:remaining} states the unresolved interfaces
explicitly.}}
\end{center}

\section{Introduction}

For a real Banach space $X$, write $B_X$ and $S_X$ for its closed unit
ball and unit sphere.  Tingley's problem asks whether every surjective
isometry
\[
 f:S_X\longrightarrow S_Y
\]
extends to a surjective real-linear isometry $X\to Y$.  A space $X$
has the \emph{Mazur--Ulam property} (MUP) if this is true for every real
Banach target $Y$.

The problem is solved for all real Banach planes
\cite{Banakh2022}, for finite-dimensional polyhedral spaces
\cite{KadetsMartin2012}, and for several further classes.  In dimension
three, strict convexity supports a useful local-to-global mechanism:
three independent distance coordinates can determine a nearby sphere
point.  Flat pieces introduce two difficulties.  First, an isometry
may initially be known only through an affine formula with a translation.
Second, the relevant chord supports may lose rank.  The rank-two loss
is geometric rather than merely algebraic: the norm becomes locally
constant in one fixed direction on an open cone.

This paper isolates that mechanism.  The main conditional result is as
follows.  The definition of a cylindrical open cone is given in
Section~\ref{sec:preliminaries}.

\begin{theorem}[Affine-anchor globalization]\label{thm:main}
Let $X$ be a three-dimensional real Banach space whose norm has no
fixed-direction cylindrical open cone.  Let $Y$ be a real Banach space
and let $f:S_X\to S_Y$ be a surjective isometry.  If there are a
nonempty relatively open set $U\subset S_X$, a linear isomorphism
$A:X\to Y$, and $c\in Y$ such that
\[
 f(x)=Ax+c\qquad(x\in U),
\]
then $c=0$, $A$ is a surjective linear isometry, and
\[
 f(x)=Ax\qquad(x\in S_X).
\]
\end{theorem}

An initial anchor can sometimes be read directly from the source ball.
For $e\in S_X$, put
\[
 \St_X(e)=\{u\in S_X:\|e+u\|=2\}.
\]

\begin{corollary}[Open-star criterion]\label{cor:open-star}
Suppose $X$ is three-dimensional, its norm has no fixed-direction
cylindrical open cone, and
\[
 \operatorname{int}_{S_X}\St_X(e)\ne\varnothing
\]
for some $e\in S_X$.  Then $X$ has the Mazur--Ulam property.
\end{corollary}

The criterion is not a disguised strict-convexity assumption.  For
example, the norm
\[
 \|(t,u)\|=|t|+\|u\|_2
 \qquad(t\in\R,\ u\in\R^2)
\]
has a sphere ruled by line segments from its two vertices.  Its ruling
direction changes along every open part of the conical surface, so no
single nonzero direction cylindricalizes an open cone.  Corollary
\ref{cor:open-star} applies.

There is also a complementary global mechanism which tolerates a full
cylinder.

\begin{theorem}[Opposite-facet prism criterion]\label{thm:prism}
Let $X$ be a finite-dimensional real Banach space.  If $F$ is a facet
of $B_X$ and
\[
 B_X=\operatorname{conv}(F\cup(-F)),
\]
then $X$ has the Mazur--Ulam property.  In particular, every
$Z\oplus_\infty\R$ with $\dim Z=2$ has the Mazur--Ulam property.
\end{theorem}

The paper ends at the exact remaining boundary.  A calibrated local
cylinder may project to an arc of the unit circle of a two-dimensional
quotient.  We show that adding the ordinary chord cones generated at a
calibrated zero-width endpoint forces the network outside any one
proper projective arc.  The new cones initially carry equality of the
two ambient norms; they do not automatically carry equality of the
sphere maps.  Section~\ref{sec:remaining} shows that the
transverse-ruled route to an initial anchor produces exactly the same
type of data.  Thus the two proof routes meet at one calibration-upgrade
problem rather than ending in two unrelated analytic gaps.

\section{Metric and convex preliminaries}\label{sec:preliminaries}

All spaces are real.  If $N$ is a norm on a finite-dimensional vector
space $X$ and $x\ne0$, let
\[
 J_N(x)=\{p\in X^*:N^*(p)=1,\ p(x)=N(x)\}
\]
be the set of norming functionals.  At a smooth point it is a singleton
and we write $DN(x)$ for its member.  We use the Euler identity
\[
 DN(x)(x)=N(x)
\]
at differentiability points.

We use four standard inputs.  A surjective sphere isometry in finite
dimensions is odd \cite{Tingley1987}.  It preserves maximal convex
subsets of the sphere \cite{Tanaka2014}.  On such a finite-dimensional
convex body it is affine by Mankiewicz's theorem
\cite{Mankiewicz1972}.  Finally, the tangent and chord-support
correspondence of Kadets and Mart\'in associates compatible linear
maps to non-smooth base points and pairs the norming functionals of
ordinary source and target chords \cite{KadetsMartin2012}.  We state
the exact consequences when they are used.

\begin{definition}[Open anchor]
An \emph{affine open anchor} for $f:S_X\to S_Y$ is a nonempty relatively
open $U\subset S_X$ on which
\[
 f(x)=Ax+c
\]
for a linear isomorphism $A:X\to Y$ and a fixed $c\in Y$.  It is a
\emph{linear open anchor} when $c=0$.
\end{definition}

\begin{definition}[Fixed-direction cylindrical open cone]
A norm $N$ on $X$ has a fixed-direction cylindrical open cone if there
are a nonzero $k\in X$ and a nonempty open cone
$\Gamma\subset X\setminus\{0\}$ such that
\begin{equation}\label{eq:cylinder-def}
 N(z+tk)=N(z)
\end{equation}
whenever the segment $[z,z+tk]$ is compactly contained in $\Gamma$.
The vector $k$ is fixed throughout $\Gamma$.
\end{definition}

Every facet germ produces such a cone: on a smaller cone the norm is a
single supporting functional $p$, and any nonzero $k\in\ker p$ works.
The converse is not asserted.

We repeatedly use the following elementary analytic observation.

\begin{lemma}[Almost-everywhere annihilation]\label{lem:acl}
Let $O\subset X$ be open, let $N$ be a norm, and let $0\ne k\in X$.
If $DN(z)(k)=0$ for almost every $z\in O$, then
\[
 N(z+tk)=N(z)
\]
whenever $[z,z+tk]$ is compactly contained in $O$.
\end{lemma}

\begin{proof}
In a rectangular box compactly contained in $O$, use $k$ as one
coordinate direction.  Rademacher's theorem and Fubini's theorem show
that the restriction of $N$ to almost every parallel line has
derivative zero almost everywhere.  The restrictions are absolutely
continuous, hence constant.  Approximation by those full-measure lines
and continuity of $N$ give the assertion on every compactly contained
parallel segment.
\end{proof}

\section{Producing an initial affine anchor}\label{sec:initial}

We first record a source-geometric entrance condition.  This is useful
both for Corollary~\ref{cor:open-star} and for locating the part of the
general problem which remains beyond the present method.

\begin{proposition}[Facet and open-star anchors]\label{prop:anchor-source}
Let $X$ be three-dimensional and let $f:S_X\to S_Y$ be a surjective
isometry.
\begin{enumerate}
\item If $B_X$ has a two-dimensional facet, then $f$ has a linear open
anchor.
\item If $\St_X(e)$ has nonempty relative interior in $S_X$, then $f$
has an affine open anchor.
\item In (2), the anchor is linear if
$\Span J_X(e)=X^*$, or if
$\St_X(e)\cap\St_X(-e)\ne\varnothing$.
\end{enumerate}
\end{proposition}

\begin{proof}
Let $F=B_X\cap\{p=1\}$ be a facet.  The restriction of $f$ to $F$ is
affine.  Choose $x_0\in\relint F$, put
$H=\Span(F-F)=\ker p$, and write
\[
 f(x_0+v)=f(x_0)+Bv\qquad(v\in H,\ x_0+v\in F).
\]
The target face has direction space $H'$, and
$Y=H'\oplus\R f(x_0)$.  Thus
\[
 L(v+tx_0)=Bv+t f(x_0)
\]
defines a linear isomorphism with $f=L$ on $F$.  The relative interior
of a facet is relatively open in $S_X$, proving (1).

For (2), assume first that $e$ is smooth, with unique support $p_e$.
If $u\in\St_X(e)$, equality in the triangle inequality supplies a
common support of $e$ and $u$, which must be $p_e$.  An open part of
$\St_X(e)$ therefore lies in the facet $\{p_e=1\}$, and (1) applies.

If $e$ is non-smooth, the finite-dimensional tangent construction in
\cite{KadetsMartin2012} gives a linear isomorphism $T_e:X\to Y$ and a
canonical defect
\[
 \delta_e=e-T_e^{-1}f(e)
\]
such that
\begin{equation}\label{eq:star-affine}
 f(u)=T_e(u-\delta_e)\qquad(u\in\St_X(e)).
\end{equation}
This gives an affine open anchor.  If the supports at $e$ span $X^*$,
the same tangent theorem gives $T_e e=f(e)$, hence $\delta_e=0$.  If
the positive and negative stars meet, their two formulas have defects
$\delta_e$ and $-\delta_e$; oddness forces them to agree at an
intersection point, again giving $\delta_e=0$.
\end{proof}

Full-support points are scarce enough for a Baire argument.

\begin{proposition}[Countability of full-support points]\label{prop:countable}
The set
\[
 V_3=\{e\in S_X:\Span J_X(e)=X^*\}
\]
is at most countable.  If a nonempty relatively open set
$W\subset S_X$ is covered by $\{\St_X(e):e\in V_3\}$, then one of these
stars has nonempty relative interior and hence produces a linear open
anchor.
\end{proposition}

\begin{proof}
For $e\in V_3$, the cone generated by $\relint J_X(e)$ is a nonempty
open subset of $X^*$.  Every functional in it uniquely exposes $e$.
The corresponding open cones are disjoint for distinct $e$.  A
second-countable space contains only countably many disjoint nonempty
open sets, proving countability.  Each star is closed in $S_X$.
Baire's theorem applied to the countable closed cover of $W$ makes one
star have interior.  Proposition~\ref{prop:anchor-source}(3) applies.
\end{proof}

\section{Propagation from a linear open anchor}\label{sec:propagation}

Assume now that $f=L$ on a nonempty relatively open $U\subset S_X$.
Put
\[
 Q(x)=\|Lx\|_Y,\qquad g=L^{-1}f:S_N\longrightarrow S_Q,
\]
where $N$ is the source norm.  Then $g$ fixes $U\cup(-U)$.  Let
\[
 \mathcal F=\{x\in S_N:g(x)=x\}.
\]

Fix $x\in\mathcal F$.  If $A$ is a relatively open part of $U$ or
$-U$ satisfying $N(x+a)<2$ for $a\in A$, define
\[
 \rho_{x,A}(a)=\frac{x-a}{N(x-a)},
 \qquad
 \Gamma(x,A)=\pos\rho_{x,A}(A).
\]

\begin{lemma}[Strict chord cones]\label{lem:strict-cone}
$\rho_{x,A}$ is an open embedding into $S_N$, and $\Gamma(x,A)$ is an
ambient open cone.  Moreover,
\begin{equation}\label{eq:two-norm-calibration}
 Q=N\qquad\hbox{on }\Gamma(x,A).
\end{equation}
\end{lemma}

\begin{proof}
If $\rho_{x,A}(a)=\rho_{x,A}(b)$, then $x-a=\lambda(x-b)$ for some
$\lambda>0$.  Unless $a=b$, one of $a,b$ is a proper convex
combination of the other and $x$.  Since all three have norm one, the
whole connecting segment lies in the sphere and has a common support;
this gives $N(x+a)=2$ or $N(x+b)=2$, a contradiction.  Thus the map is
injective.  Invariance of domain on the two-dimensional norm spheres
makes its image open.  Finally, $x$ and $a$ are fixed, so
\[
 Q(x-a)=Q(g(x)-g(a))=N(x-a),
\]
and homogeneity proves \eqref{eq:two-norm-calibration}.
\end{proof}

Let $\mathcal G_x$ be the span in $X^*$ of all $DN(d)$ at
differentiability points $d$ in all available cones $\Gamma(x,A)$.

\begin{proposition}[Distance-coordinate trichotomy]\label{prop:trichotomy}
At a fixed point $x\in\mathcal F$ one of the following holds.
\begin{enumerate}
\item $\dim\mathcal G_x=3$, and $x$ is an interior point of
$\mathcal F$ in $S_N$.
\item $\dim\mathcal G_x=2$, and all available strict chord cones share
a nonzero fixed cylindrical direction.
\item $\dim\mathcal G_x=1$, and an available chord cone is a facet cone.
\item One sign of the anchor supplies no strict open chord family, in
which case $U$ or $-U$ is contained in $\St_X(x)$.
\end{enumerate}
\end{proposition}

\begin{proof}
If the span has dimension three, choose three anchor points $a_i$ whose
chord directions $d_i=x-a_i$ are smooth and whose supports
$p_i=DN(d_i)$ are independent.  The map
$z\mapsto(p_1z,p_2z,p_3z)$ is injective.  Compactness of the unit sphere
and upper semicontinuity of the subdifferential give, in a neighborhood
of $x$,
\[
 cN(y-z)\le
 \max_i|N(y-a_i)-N(z-a_i)|
\]
for some $c>0$.  Shrink the neighborhood so that both $y-a_i$ and
$g(y)-a_i$ remain in the calibrated cone belonging to $d_i$.  Sphere
isometry and \eqref{eq:two-norm-calibration} give
\[
 N(y-a_i)=Q(g(y)-a_i)=N(g(y)-a_i).
\]
The coordinate inequality yields $g(y)=y$, proving (1).

If the span has dimension at most two, choose
$0\ne k\in\mathcal G_x^\perp$.  The gradients of $N$ annihilate $k$
almost everywhere in every available cone.  Lemma~\ref{lem:acl}
gives the cylinder identity.  If the span is one-dimensional, all
gradients on a connected subcone equal one fixed support, so the norm
is linear there and the cone is a facet cone.  Finally, if no strict
open family exists for one sign, openness of strict inequality shows
$N(x+a)=2$ throughout that signed anchor, which is precisely the
face-star assertion.
\end{proof}

\section{Open face-stars and full-support centers}\label{sec:facestar}

The face-star alternative in Proposition~\ref{prop:trichotomy} has a
second rank decomposition.

\begin{proposition}[Face-star trichotomy]\label{prop:face-star}
Suppose $x\in S_N$, $V\subset S_N$ is nonempty and relatively open,
and $[x,v]\subset S_N$ for every $v\in V$.  Then at least one of the
following occurs:
\begin{enumerate}
\item $V$ lies in a facet germ;
\item $\pos V$ is a fixed-direction cylindrical open cone;
\item $x$ is full-support: $\Span J_N(x)=X^*$.
\end{enumerate}
\end{proposition}

\begin{proof}
At every smooth $v\in V$, the unique support $p_v=DN(v)$ also norms
$x$, because a support at an interior point of $[x,v]$ must norm both
endpoints.  Let $M=\Span\{p_v:v\in V\text{ smooth}\}$.  Smooth points
have full surface measure.  If $\dim M=1$, continuity puts $V$ into
one exposed facet.  If $\dim M=2$, a nonzero $k\in M^\perp$ is
annihilated by the gradient almost everywhere in $\pos V$, and
Lemma~\ref{lem:acl} gives (2).  If $\dim M=3$, then
$\Span J_N(x)=X^*$.  In particular, the supports at $x$ separate
directions, so $x$ is an extreme point.
\end{proof}

We next state the local full-support rigidity which completes the
closed-open propagation.  Its proof is included because this is where
the non-strictly-convex case differs from the usual three-coordinate
argument.

\begin{proposition}[Full-support alternative]\label{prop:full-support}
Let $e\in S_N$ satisfy $\Span J_N(e)=X^*$, and let $T_e:X\to Y$ be its
canonical tangent isomorphism.  On every connected relatively open
patch
\[
 W\Subset\{x\in S_N:N(x-e)<2,\ N(x+e)<2\},
\]
one has either $f(x)=T_e x$ for all $x\in W$, or the source norm has a
fixed-direction cylindrical open cone.
\end{proposition}

\begin{proof}
The tangent theorem gives $f(\pm e)=T_e(\pm e)$.  At almost every
$x\in W$, the four source and target chord directions are smooth.
Write
\[
 p_\pm=DN(x\pm e),
 \qquad p'_\pm=D\|\cdot\|_Y(f(x)\pm f(e)).
\]
The chord-support correspondence gives
\begin{equation}\label{eq:support-pair}
 p'_\pm\circ T_e=p_\pm.
\end{equation}
For $d(x)=f(x)-T_ex$, equality of the two endpoint distances and
\eqref{eq:support-pair} imply
\[
 p'_-(d(x))=p'_+(d(x))=0.
\]
The two target supports are independent.  Indeed, a negative
proportionality would force equality in the triangle inequality at
$f(x)$ and contradict the uniqueness of the smooth chord supports; a
positive proportionality would express the full-support extreme point
$-f(e)$ as a proper convex combination of two distinct sphere points.
Thus $d(x)$ lies in the one-dimensional line
\[
 K'(x)=\ker p'_-\cap\ker p'_+.
\]

Differentiate the exact identities
\[
 \|f(x)\pm f(e)\|_Y=N(x\pm e)
\]
in sphere-chart tangent directions.  Subtract
\eqref{eq:support-pair}.  At almost every differentiability point,
\[
 Dd(x)(T_xS_N)\subset K'(x).
\]
On the open set where $d(x)\ne0$, the line $K'(x)$ equals
$\R d(x)$.  In a local chart choose a continuous linear functional
$\lambda$ which does not vanish on $d$ and normalize
$u(x)=d(x)/\lambda(d(x))$.  The preceding differential inclusion
gives $Du=0$ almost everywhere.  Absolute continuity on lines makes
$u$ constant on each connected component.  Consequently the kernel
line is fixed there.  Applying the source version of
\eqref{eq:support-pair} and Lemma~\ref{lem:acl} to either family of
chord cones gives a fixed-direction cylindrical open cone.  If the
nonzero set is empty, $f=T_e$ on $W$.
\end{proof}

\begin{proof}[Proof of the linear-anchor part of Theorem~\ref{thm:main}]
Assume first that $c=0$.  The fixed set $\mathcal F$ contains
$U\cup(-U)$ and is closed.  Let $x\in\mathcal F$.  The rank-two and
rank-one alternatives of Proposition~\ref{prop:trichotomy} contradict
the no-cylinder hypothesis.  In the remaining face-star alternative,
Proposition~\ref{prop:face-star} removes its facet and cylindrical
branches.  Hence $x$ is full-support.

The tangent map at $x$ agrees with $L$ on the signed anchor contained
in $\St_X(x)$.  If the difference span of that open patch has dimension
three, the two linear maps agree everywhere.  If it has dimension two,
the patch is planar and is a facet germ, already excluded.  Thus
$T_x=L$.  Near $x$, points in $\St_X(x)$ are fixed by the star formula;
points in the double strict region are fixed by
Proposition~\ref{prop:full-support}, since its cylinder alternative is
excluded.  Therefore $x$ is interior to $\mathcal F$.

The fixed set is nonempty, closed, and open in the connected sphere,
so it is the whole sphere.  It follows that $f=L$ on $S_X$ and that
$\|Lz\|_Y=N(z)$ for every $z\in X$ by homogeneity.  Surjectivity of
$f$ makes $L$ surjective.
\end{proof}

\section{Removing the translation from an affine anchor}\label{sec:translation}

The translation can be detected by comparing the anchor with its
antipodal copy.

\begin{proposition}[Translation forces a cylinder]\label{prop:translation}
Let $U\subset S_N$ be a nonempty relatively open set which is not a
facet germ, and suppose
\[
 f(x)=Ax+c\qquad(x\in U)
\]
with $A$ invertible.  If $c\ne0$, the source norm $N$ has a
fixed-direction cylindrical open cone.
\end{proposition}

\begin{proof}
Put $Q=P\circ A$ and $d=A^{-1}c$.  Oddness gives, for $x,y\in U$,
\begin{equation}\label{eq:sum-distance}
 N(x+y)=P(f(x)-f(-y))=Q(x+y+2d).
\end{equation}
Because $U$ is not a facet germ, it contains two smooth points
$x_0,y_0$ with distinct tangent planes.  Local Lipschitz graph
parametrizations $\phi,\psi$ at those points have tangent images whose
sum is $X$.  Restricting one parameter of $\psi$, the map
\[
 (s_1,s_2,t)\longmapsto\phi(s_1,s_2)+\psi(t)
\]
has an invertible derivative at the origin.  The elementary local
degree theorem therefore shows that $U+U$ contains an ambient open set
$O$.

On $O$, equation \eqref{eq:sum-distance} reads
$Q(z+2d)=N(z)$.  At almost every point both sides are differentiable
and have the same derivative $p_z$.  Euler's identity applied to the
two norms gives
\[
 p_z(z+2d)=Q(z+2d)=N(z)=p_z(z),
\]
so $DN(z)(d)=0$ almost everywhere on $O$.  If $d\ne0$,
Lemma~\ref{lem:acl} gives a cylinder in direction $d$.  Since $A$ is
injective, $c\ne0$ is equivalent to $d\ne0$.
\end{proof}

\begin{proof}[Completion of Theorem~\ref{thm:main}]
The no-cylinder hypothesis also excludes facet germs.  Hence the
anchor patch is not a facet.  Proposition~\ref{prop:translation}
forces $c=0$, and the linear-anchor argument in
Section~\ref{sec:facestar} completes the proof.
\end{proof}

\begin{proof}[Proof of Corollary~\ref{cor:open-star}]
Proposition~\ref{prop:anchor-source} produces an affine open anchor.
Theorem~\ref{thm:main} globalizes it.
\end{proof}

\begin{corollary}[A non-strictly-convex example]\label{cor:double-cone}
The space $\R\oplus_1\ell_2^2$ has the Mazur--Ulam property.
\end{corollary}

\begin{proof}
At $e=(1,0)$, the norming set
\[
 J(e)=\{(1,b):\|b\|_2\le1\}
\]
spans the dual, and the upper conical surface is contained in
$\St_X(e)$.  Thus an initial linear anchor exists.  On the smooth
region $t\ne0$, $u\ne0$, the gradient is
\[
 \left(\operatorname{sgn}t,\frac{u}{\|u\|_2}\right).
\]
On every ambient open cone, $u/\|u\|_2$ sweeps a nondegenerate circular
arc.  A fixed vector annihilated by all these gradients must be zero.
Lemma~\ref{lem:acl} therefore rules out a fixed-direction cylindrical
open cone.  Corollary~\ref{cor:open-star} applies.
\end{proof}

\section{Spherical segment saturation and prisms}\label{sec:saturation}

The next argument handles some balls which are cylindrical everywhere.
For $E\subset S_X$, define
\[
 \mathcal H(E)=E\cup
 \bigcup\{[a,b]:a,b\in E,\ [a,b]\subset S_X\}.
\]
Iterate this operation and take closures.  Call $E$ \emph{spherically
segment-generating} if the smallest closed, segment-saturated subset
of $S_X$ containing $E$ is $S_X$.

\begin{proposition}[Segment saturation]\label{prop:saturation}
Suppose $f=L$ on $E\subset S_X$, where $L:X\to Y$ is linear.  Then
$f=L$ on the closed spherical segment saturation of $E$.  In
particular, if $E$ spherically segment-generates $S_X$, then $f=L$ on
$S_X$ and $L$ is the unique surjective linear isometric extension.
\end{proposition}

\begin{proof}
If $[a,b]\subset S_X$ and its endpoints are calibrated, place the
segment in a maximal convex subset $M$ of $S_X$.  The restriction
$f|_M$ is affine, so
\[
 f((1-t)a+tb)=(1-t)f(a)+tf(b)=L((1-t)a+tb).
\]
Induction and continuity propagate the identity through all iterates
and their closure.  The final statements follow from homogeneity and
surjectivity.
\end{proof}

\begin{proof}[Proof of Theorem~\ref{thm:prism}]
The restriction of a sphere isometry to the facet $F$ is affine and
therefore agrees with a linear isomorphism $L$ on $F$; oddness gives
the same on $-F$.  We show that $F\cup(-F)$ segment-generates $S_X$ in
one step.

Let $x\in S_X$ and choose $p\in J_X(x)$.  Express $x$ as a convex
combination of points of $F\cup(-F)$.  Since $p(x)=1$ and $p\le1$ on
$B_X$, every point carrying positive weight lies in $\{p=1\}$.  Taking
the weighted barycenter separately in $F$ and in $-F$ gives
\[
 x=\lambda a+(1-\lambda)b,
 \qquad a\in F\cap\{p=1\},\quad
 b\in(-F)\cap\{p=1\}.
\]
Thus $[a,b]\subset S_X$, proving the saturation claim.  For
$X=Z\oplus_\infty\R$, take
$F=B_Z\times\{1\}$; then
$B_X=\operatorname{conv}(F\cup(-F))$.
\end{proof}

\section{Escape from a proper quotient arc}\label{sec:arc}

We now isolate one rigorous consequence inside the cylinder regime.
Let $e\ne0$, let $\pi:X\to Z=X/\R e$, and let $n$ be the quotient
norm.  An open cone $\mathcal C\subset X\setminus\{0\}$ is called an
\emph{$L$-visible $e$-cylindrical chord cone} when
\[
 N(w+te)=N(w)
 \quad([w,w+te]\Subset\mathcal C),
 \qquad
 \|Lw\|_Y=N(w)\quad(w\in\mathcal C).
\]

\begin{theorem}[Proper-arc escape]\label{thm:arc-escape}
Let $I\subset S_n$ be an open arc with linearly independent endpoint
representatives $a_0,a_1$, and suppose its projectivization is contained
in the proper closed projective arc determined by
\[
 C=\{\alpha a_0+\beta a_1:\alpha,\beta\ge0\}.
\]
Assume that:
\begin{enumerate}
\item a relatively open calibrated set $U\subset S_X$ consists of the
interiors of nondegenerate complete $e$-fibers over $I$;
\item the fiber over $a_0$ degenerates to a calibrated point $x_0$;
\item at $x_0$ the ordinary distance-coordinate exit has rank two with
common kernel $\R e$, and for every calibrated strict chord endpoint
$z$ under consideration one has
$N(x_0-z)=n(a_0-\pi z)$; and
\item no open face-star occurs at $x_0$.
\end{enumerate}
Then chord saturation at $x_0$ produces $L$-visible
$e$-cylindrical chord cones whose projective quotient directions leave
the original closed projective arc.
\end{theorem}

\begin{proof}
The rank-two endpoint identity says that, for every calibrated point
$z$ joined to $x_0$ by a strict ordinary chord,
\begin{equation}\label{eq:quotient-distance}
 N(x_0-z)=n(a_0-\pi z).
\end{equation}
If the set of such $z$ were not dense in $U$, an open subpatch would
satisfy $N(x_0+z)=2$ and would form an open face-star.  Hence choose
open patches $W_j$ of strict points whose quotient projections contain
$q_j\to a_1$.

Set
\[
 \mathcal C_j=\{t(x_0-z):t>0,\ z\in W_j\}.
\]
Ordinary chord-direction invariance of domain makes $\mathcal C_j$ an
ambient open cone.  Calibration at both endpoints and sphere isometry
give
\[
 \|L(x_0-z)\|_Y=N(x_0-z).
\]
Because $z$ lies in the interior of a complete $e$-fiber, it may be
moved a small distance in direction $e$ without changing $\pi z$.
Equation \eqref{eq:quotient-distance} then gives the cylinder identity
on $\mathcal C_j$.  Thus every $\mathcal C_j$ is $L$-visible.

Its quotient directions contain
\[
 [\pi(x_0-z_j)]=[a_0-q_j]\longrightarrow[a_0-a_1].
\]
But
\begin{equation}\label{eq:outside-cone}
 [a_0-a_1]\notin\PP(C).
\end{equation}
Indeed, an identity
$\lambda(a_0-a_1)=\alpha a_0+\beta a_1$ with
$\alpha,\beta\ge0$ would imply
$\alpha=\lambda$ and $\beta=-\lambda$.  This is impossible for
$\lambda\ne0$.  Since $\PP(C)$ is compact, the quotient direction of
$\mathcal C_j$ lies outside it for all sufficiently large $j$.
\end{proof}

\begin{remark}[What Theorem~\ref{thm:arc-escape} does not say]
The conclusion is equality of the two norms on new cones.  It does not
assert
\[
 f\left(\frac{w}{N(w)}\right)
 =\frac{Lw}{N(w)}.
\]
Thus the escaped cone is not automatically a new sphere-map anchor.
It may land in another quotient arc.  The theorem excludes a terminal
network confined to one proper projective arc, but it does not exclude
a network made of several arcs.
\end{remark}

\section{The remaining affine-anchor problem}\label{sec:remaining}

Theorem~\ref{thm:main} begins with an affine open anchor, while
Proposition~\ref{prop:anchor-source} produces one from a facet or an
open cofacial star.  These source conditions need not be automatic.
The current convex-geometric reduction is the following.

Fix an auxiliary Euclidean norm $|\cdot|_2$ and define closed subsets
of the sphere by
\[
 \mathcal M_m=\{x\in S_X:\exists v,\ |v|_2=1,\
 x\pm m^{-1}v\in S_X\}.
\]
The nonextreme points are exactly $\bigcup_m\mathcal M_m$.  Baire's
theorem yields a dichotomy.

\begin{proposition}[Unanchored geometric dichotomy]\label{prop:unanchored}
Either $\Ext(B_X)$ is a dense $G_\delta$ subset of $S_X$, or there is a
nonempty relatively open $V\subset S_X$ and $\varepsilon>0$ such that
every $x\in V$ is the midpoint of a spherical segment of Euclidean
half-length at least $\varepsilon$.  If $B_X$ has no facet, the line
direction through $x$ is unique, varies continuously with $x$, and is
constant along each local ruling segment.
\end{proposition}

\begin{proof}
Each $\mathcal M_m$ is closed by compactness of the Euclidean unit
sphere.  If all have empty interior, their union is meagre and the
extreme points form a dense $G_\delta$.  Otherwise an open subset of
one $\mathcal M_m$ gives the uniform segment length.  In the absence of
a facet, two nonparallel spherical segments through the same point
would generate a planar face germ, so the direction is unique.  Limits
of the uniformly long segments prove continuity, and uniqueness makes
the direction constant along each segment.
\end{proof}

Removing the affine-anchor hypothesis therefore requires one of two
new arguments:
\begin{enumerate}
\item reconstruct local distance coordinates on a dense extreme
skeleton without assuming local strict convexity; or
\item prove that a uniformly ruled patch with continuous varying
direction either contains an open cofacial star, becomes a fixed
cylinder, or directly determines an affine formula for the sphere
isometry.
\end{enumerate}

The second branch admits one further reduction.  It is the precise
analogue, before an open anchor exists, of the norm cones appearing in
Theorem~\ref{thm:arc-escape}.

\begin{proposition}[Two transverse rulings]\label{prop:two-rulings}
Suppose the ruled patch contains two nondegenerate spherical segments
\[
 \Sigma_i=\{x_i+t k_i:t\in I_i\}\subset S_X,
 \qquad 0\in I_i,\quad i=0,1,
\]
and put $d=x_0-x_1$.  If $d,k_0,k_1$ are linearly independent, then
there are a linear map $A:X\to Y$, an affine map $B:X\to Y$, and a
nonempty ambient open cone $\Gamma\subset X\setminus\{0\}$ such that
\[
 \|Aw\|_Y=N(w)\quad(w\in\Gamma),
 \qquad f=B\quad\hbox{on }\Sigma_0\cup\Sigma_1.
\]
If $N$ has no fixed-direction cylindrical open cone, then $A$ is a
linear isomorphism.
\end{proposition}

\begin{proof}
Affineness on maximal convex spherical subsets gives
\[
 f(x_i+t k_i)=y_i+t\ell_i\qquad(t\in I_i).
\]
Set $d'=y_0-y_1$ and define $A$ by
$Ad=d'$, $Ak_0=\ell_0$, and $Ak_1=\ell_1$.  Every cross-segment
distance gives
\[
 N(d+t k_0-s k_1)
 =\|A(d+t k_0-s k_1)\|_Y.
\]
The map
\[
 (r,t,s)\longmapsto r(d+t k_0-s k_1)
\]
has full-rank derivative, so the positive hull of these differences is
an ambient open cone $\Gamma$.  Homogeneity gives the asserted norm
identity on $\Gamma$.  If $0\ne k\in\ker A$, that identity implies
$N(w+uk)=N(w)$ on every sufficiently short segment contained in
$\Gamma$, producing a forbidden fixed-direction cylinder.  Hence $A$
is injective under the no-cylinder hypothesis.  Finally,
\[
 B(z)=A(z-x_1)+y_1
\]
agrees with $f$ on both segments by construction.
\end{proof}

Proposition~\ref{prop:two-rulings} does not yet produce an affine open
anchor: two line segments have empty relative interior in the sphere.
It produces the same pair of ingredients seen after the proper-arc
escape: an ambient norm-calibrated cone and a lower-dimensional set on
which the sphere map is genuinely calibrated.

\begin{problem}[The calibration-upgrade obstruction]\label{prob:upgrade}
Find intrinsic hypotheses on an ambient open cone $\Gamma$, a linear
isomorphism $A$, and a calibrated lower-dimensional spherical set $E$
which force
\[
 \|Aw\|_Y=N(w)\quad(w\in\Gamma),
 \qquad f(x)=Ax+c\quad(x\in E)
\]
to imply
\[
 f(x)=Ax+c
\]
on a nonempty relatively open subset of $S_X$.
\end{problem}

This is the single common analytic obstruction reached by the two
routes developed here.  In the multi-arc cylinder route,
Theorem~\ref{thm:arc-escape} supplies escaped cones with
$\|Lw\|_Y=N(w)$, while calibrated endpoint fibers supply $E$.  In the
transverse-ruled route, Proposition~\ref{prop:two-rulings} supplies
$\Gamma,A$, and two calibrated rulings.  Solving
Problem~\ref{prob:upgrade} in either geometry would turn the new norm
cone into an actual open anchor, after which the propagation results of
Sections~\ref{sec:propagation}--\ref{sec:translation} apply.

There are still geometric entrance tasks: a multi-arc network must be
organized so that enough endpoint data survive, and a continuously
ruled patch with no transverse pair must be classified into facet,
cylinder, common-vertex, or tangential-developable behavior.  The dense
extreme branch of Proposition~\ref{prop:unanchored} also requires a
separate entrance argument.  These tasks determine how one reaches the
data of Problem~\ref{prob:upgrade}; once those data are reached, the
remaining logical gap is the one calibration upgrade stated above.
Consequently, the results of this paper do not decide Tingley's problem
for an arbitrary three-dimensional real Banach space.

\section*{Acknowledgement of proof status}

The arguments in Sections~\ref{sec:propagation}--\ref{sec:arc} were
assembled as part of an ongoing proof search.  Before submission, the
tangent-map dependencies in Propositions~\ref{prop:anchor-source} and
\ref{prop:full-support}, the low-regularity differentiation step in
Proposition~\ref{prop:full-support}, the endpoint identity
\eqref{eq:quotient-distance}, and
Proposition~\ref{prop:two-rulings} should receive an independent
line-by-line audit.  The conditional statements are deliberately
separated from Problem~\ref{prob:upgrade} so that no open branch is used
as a theorem.

\end{document}